\documentclass{amsart}

\usepackage{graphicx}
\usepackage[utf8]{inputenc}
\usepackage[english]{babel}
\usepackage{csquotes}
\usepackage{amssymb}
\usepackage{tikz-cd}

\usepackage{cite}
\usepackage{amsmath,amsfonts,pictex,graphicx,fullpage,etex,tikz,hyperref}
\usepackage{enumerate}
\usepackage{enumitem}
\numberwithin{equation}{section}
\newcounter{dummy} \numberwithin{dummy}{section}

\newtheorem{definition}[dummy]{Definition}
\newtheorem{lemma}[dummy]{Lemma}
\newtheorem{theorem}[dummy]{Theorem}
\newtheorem{corollary}[dummy]{Corollary}
\newtheorem{proposition}[dummy]{Proposition}
\newtheorem{*remark}[dummy]{Remark}

\newtheorem{notation}[dummy]{Notations}

\newtheorem*{claim}{Claim}
\makeatletter

\newcommand{\M}{\mathrm{M}}
\newcommand{\R}{\mathbb{R}}
\newcommand{\Z}{\mathbb{Z}}

\newcommand{\x}{\mathbf{x}}

\newcommand{\I}{\mathrm{I}}

\newcommand{\dd}{\mathrm{d}}
\makeatother

\begin{document}

\title{Equidistribution of expanding curves in homogeneous spaces and Diophantine approximation for square matrices}

\author{Lei Yang $^{\ast}$}
\address{Mathematical Sciences Research Institute, Berkeley, CA, 94720, U.S.A.}

\curraddr{Einstein Institute of Mathematics, Hebrew University of Jerusalem, Jerusalem, 9190401, Israel}
\email{yang.lei@mail.huji.ac.il}
\thanks{$^{\ast }$ Supported in part by a Postdoctoral Fellowship at MSRI}



\date{}



\begin{abstract}
In this article, we study an analytic curve $\varphi: I=[a,b]\rightarrow \mathrm{M}(n\times n, \R)$ 
in the space of $n$ by $n$ real matrices, and show that if $\varphi$ satisfies certain geometric conditions, then 
for almost every point on the curve, the Diophantine approximation given by Dirichlet's Theorem is not improvable. To do this, 
we embed the curve into some homogeneous space $G/\Gamma$, and prove that under the action of some expanding 
diagonal flow $A= \{a(t): t \in \R\}$, the expanding curves tend to be equidistributed in $G/\Gamma$, as $t \rightarrow +\infty$. This 
solves a special case of a problem proposed by Nimish Shah in ~\cite{Shah_1}.
\end{abstract}

\maketitle


\section{Introduction}
\subsection{Dirichlet's Theorem on Diophantine approximation}
For any real vector space $\R^k$ and $\x \in \R^k$, let $\|x\|$ denote the maximal norm of $\x$, i.e., if 
$\x =(x_1,x_2,\dots, x_k)$,  $\|\x\|:=\max_{1\leq i\leq k}|x_i|$.  
Given two positive integers $m$ and $n$, and a $m$ by $n$ matrix $\Phi \in \mathrm{M}(m\times n, \R)$, concerning Diophantine approximation property of $\Phi$, we have the following well known theorem due to Dirichlet:
\begin{theorem}[Dirichlet's Theorem]

 Given any $\Phi \in \mathrm{M}(m\times n, \R)$ and any large $N>0$, there exists nonzero integer vector $\mathbf{p} \in \Z^n$ with 
 $\|\mathbf{p}\| \leq N^m$ and integer vector $\mathbf{q} \in \Z^m$ such that $\|\Phi \mathbf{p} -\mathbf{q}\| \leq N^{-n}$.
 
\end{theorem}

\par Now we consider the following finer question: for a particular $m$ by $n$ matrix $\Phi$, could we improve Dirichlet's Theorem? By improving Dirichlet's 
Theorem, we mean there exists a constant $0 < \mu <1$, such that for all large $N > 0$, there exists nonzero integer vector $\mathbf{p} \in \Z^n$ with
$\|\mathbf{p}\| \leq \mu N^m$, and integer vector $\mathbf{q} \in \Z^m$ such that $\|\Phi \mathbf{p} - \mathbf{q}\| \leq \mu N^{-n}$. 
If such constant
$\mu$ exists, then we say $\Phi$ is $DT_{\mu}$-improvable. And if $\Phi$ is $DT_{\mu}$-improvable for some $0<\mu <1$, then we say $\Phi$ is 
$DT$-improvable (here $DT$ stands for Dirichlet's Theorem). 
\par In 1970, Davenport and Schmidt ~\cite{Daven_Schm} considered this question and answered it to some extend: they proved that almost every matrix $\Phi \in \M(m\times n,\R)$ is not $DT$-improvable. In ~\cite{Daven_Schm}, it was also proved that for $m=1$ and $n=2$, $\M(1\times 2, \R) \cong \R^2$, almost every point on the curve 
$$\phi(s)=(s,s^2):s \in \R$$ is not $DT_{1/4}$ improvable. This result for the particular curve $\phi$ was generalized by Baker ~\cite{Baker}: it was proved that for any 
smooth curve in $\R^2$ satisfying some curvature condition, almost every point on the curve is not $DT_{\mu}$ improvable for some $0<\mu <1$ depending on the curve. In 2002, Bugeaud ~\cite{Bugeaud} generalized the result of Davenport and Schmidt in the following sense: For $m=1$, and general 
$n$, almost every point on the curve $\varphi(s)=(s,s^2,\dots, s^n)$ is not $DT_{\mu}$-improvable for some small constant $0<\mu<1$. Their proofs 
are based on the technique of {\bf regular systems} introduced by Davenport and Schmidt ~\cite{Daven_Schm}.
\par Recently, Dani ~\cite{Dani}, and Kleinbock and Margulis ~\cite{Klein_Mar} established an elegant correspondence between Diophantine approximation and homogeneous dynamics. Based on this correspondence, Kleinbock and Weiss ~\cite{Klein_Weiss} studied this 
Diophantine approximation problem in the language of homogeneous dynamics, and proved the following result: For $m=1$ and arbitrary $n$, if an analytic curve in $\M(1\times n,\R) \cong \R^n$ satisfies some non-degenerate condition, then almost every 
point on the curve is not $DT_{\mu}$-improvable for some small constant $0<\mu<1$ depending on the curve. In 2009, Nimish Shah ~\cite{Shah_2} proved the following stronger result: For $m=1$ and general $n$, if an analytic curve 
$\varphi: I=[a,b] \rightarrow \R^n$ is not contained in any proper affine subspace, then almost every point on the curve is not $DT_{\mu}$-improvable
for any $0<\mu<1$, i.e., almost every point on the curve is not $DT$-improvable.
\par In this article, we will deal with the case of square matrices:
\begin{theorem}
\label{theorem_in_diophantine_part}
For $m=n$, if an analytic curve
$$\varphi: I=[a,b]\rightarrow \M(n\times n, \R)$$
satisfies the following condition: 
\begin{enumerate}[label=\textbf{A.\arabic*}]
 \item \label{1st} Its derivative $\varphi^{(1)}(s)$ is invertible at some $s \in I$.
 \item \label{2nd} There exist $s_0 \in I$ and a subinterval $J_{s_0} \subset I$ such that $\varphi(s)-\varphi(s_0)$ is invertible for $s \in J_{s_0}$, and moreover, 
 $\{(\varphi(s)-\varphi(s_0))^{-1}: s\in J_{s_0}\}$ is not contained in any proper affine subspace of $\M(n\times n, \R)$.
\end{enumerate}
Then almost every point on the curve is not $DT$-improvable.
\end{theorem}
\begin{*remark} $\quad$
\par 1. Since $\varphi$ is analytic, if its derivative $\varphi^{(1)}(s)$ is invertible at some point, then it must be invertible at every point except finite many points.
\par 2. If $\varphi$ satisfies condition \ref{1st} above, then for every $s_0 \in I$ such that $\varphi^{(1)}(s_0)$ is invertible, there always exists some subinterval $J_{s_0}$ of $I$, such 
 that $\varphi(s)-\varphi(s_0)$ is invertible for all $s \in J_{s_0}$. This is because when $s$ is close to $s_0$, the major part of $\varphi(s) - \varphi(s_0)$ is $(s-s_0)\varphi^{(1)}(s_0)$ which is invertible. Thus the essential part of condition \ref{2nd} is that 
 $\{(\varphi(s)- \varphi(s_0))^{-1}: s \in J\}$ is not contained in a proper affine subspace of $\M(n\times n, \R)$. 
\end{*remark}
\subsection{Equidistribution of expanding curves on homogeneous spaces}
\par Theorem \ref{theorem_in_diophantine_part} follows from an equidistribution result in homogeneous dynamics, together with the correspondence between Diophantine approximation and homogeneous dynamics. 
\par Now let us briefly recall the correspondence as follows.
\par Let $G= \mathrm{SL}(m+n,\R)$, and let $\Gamma = \mathrm{SL}(m+n,\Z)$. Then $G/\Gamma$ denotes the space of unimodular lattices of $\R^{m+n}$. 
Every point $g\Gamma$ corresponds to the unimodular lattice $g\Z^{m+n}$. For $r>0$, let $B_{r}$ denote the 
ball in $\R^{m+n}$ centered at the origin and of radius $r$. For any $0<\mu<1$, the subset 
$$K_{\mu}:= \{\Lambda \in G/\Gamma: \Lambda\cap B_{\mu} = \{\mathbf{0}\}\}$$
contains an open neighborhood of $\Z^{m+n}$ in $G/\Gamma$.
Let us define the diagonal subgroup $A=\{a(t): t \in \R\}$ by
$$a(t):= \begin{bmatrix}e^{nt} \I_m & \\ & e^{-mt}\I_n \end{bmatrix}.$$
\par Now we consider the embedding 
              $$u: \M(m\times n, \R) \rightarrow \mathrm{SL}(m+n,\R)$$
              
              $$ \Phi \in \M(m\times n, \R) \mapsto u(\Phi):= \begin{bmatrix}\I_m & \Phi \\ & \I_n \end{bmatrix}.$$
   Suppose for some $0<\mu<1$, and any $N>0$ large enough, there exist nonzero integer vector $\mathbf{p} \in \Z^n$ and integer 
   vector $\mathbf{q}\in \Z^m$ such that $\|\mathbf{p}\| \leq \mu N^m$ and $\|\Phi \mathbf{p} -\mathbf{q}\| \leq \mu N^{-n}$. 
   Then direct calculation shows that the lattice $a(\log N)u(\Phi)\Z^{m+n}$ has a vector $a(\log N)u(\Phi)(-\mathbf{q}, \mathbf{p})$ whose norm is $\leq \mu$, i.e.,
   $a(\log N) u(\Phi) \Z^{m+n} \not\in K_{\mu}$ for all $N>0$ large enough. Thus, to show that $\Phi \in \M(m\times n, \R)$ is not 
   $DT_{\mu}$-improvable,
   it suffices to show that the trajectory $\{a(t)u(\Phi)[e]: t >0\}$ meets $K_{\mu}$ infinitely many times. In particular, for an analytic curve 
   $$\varphi: I=[a,b]\rightarrow \M(m\times n, \R)$$ if we could show that for almost every point $\varphi(s)$ on the curve the trajectory 
   \begin{equation}\label{equ_ae_orbit_dense}\{a(t)u(\varphi(s))[e]: t >0\} \text{ is dense, } \end{equation}
   then we could conclude that almost every $\varphi(s)$ is not $DT$-improvable.
   It turns out that in the case $m=n$, we could prove the following much stronger result:
\begin{theorem}
 \label{goal_thm}
 Let $G$ be a Lie group containing $H = \mathrm{SL}(2n,\R)$, and $\Gamma < G$ be a lattice of $G$. Let $\mu_G$ denote the unique $G$-invariant probability 
 measure on the homogeneous space $G/\Gamma$. Take $x = g\Gamma \in G/\Gamma$ such that its $H$-orbit $Hx$ is dense in $G/\Gamma$. Let us fix the 
 diagonal flow
 $$A = \left\{a(t) = \begin{bmatrix}e^{t}\I_n & \\ & e^{-t}\I_n\end{bmatrix}\right\}.$$
 Let $\varphi: I=[a,b]\rightarrow \M(n\times n, \R)$ be an analytic curve, and embed the curve into $H$ via 
 $$u: X \in \M(n\times n , \R) \mapsto u(X) = \begin{bmatrix}\I_n & X \\ & \I_n\end{bmatrix}.$$
 Let $\mu_t$ denote the normalized Lebesgue measure on the curve $a(t)u(\varphi(I))x \subset G/\Gamma$, i.e., for a 
 compactly supported continuous function $f \in C_c(G/\Gamma)$,
 $$\int f \dd \mu_t := \frac{1}{|I|}\int_{s\in I} f(a(t)u(\varphi(s))x)\dd s.$$
 If the curve $\varphi$ satisfies the conditions \ref{1st} and \ref{2nd} given in Theorem \ref{theorem_in_diophantine_part}, 
 then $\mu_t \rightarrow \mu_G$ as $t \rightarrow +\infty$ in weak-$\ast$ topology, i.e., for any function $f\in C_c(G/\Gamma)$,
 $$\lim_{t\rightarrow +\infty} \frac{1}{|I|} \int_{s \in I} f(a(t)u(\varphi(s))x)\dd s = \int_{G/\Gamma} f \dd \mu_G.$$
\end{theorem}
\begin{*remark} $\quad$
\par 1. The assumption that $Hx$ is dense in $G/\Gamma$ does not reduce the generality of the theorem. In fact, since $H$ is generated by unipotent subgroups contained in $H$, by Ratner's theorem (cf. ~\cite{Ratner}), the closure of $Hx$ must be some homogeneous subspace $F x$ where $F$ is some Lie subgroup of $G$ containing $H$ such that the orbit $Fx$ is closed. Then we can make the assumption hold by replacing $G$, $\Gamma$ and $x = g \Gamma$ by $F$, $g\Gamma g^{-1}$ and $[e]$ respectively.
\par 2. To prove Theorem \ref{theorem_in_diophantine_part}, we only need the above theorem with $G=H=\mathrm{SL}(2n,\R)$, $\Gamma = \mathrm{SL}(2n,\Z)$,
 and $x = [e] = \Z^{2n} \in G/\Gamma$, since the above equidistribution result 
 immediately implies that for almost every $s \in I$, (\ref{equ_ae_orbit_dense}) holds (see ~\cite{Shah_2} for details). 
\par 3. Even in the case $G=H=\mathrm{SL}(2n,\R)$, Theorem \ref{goal_thm} is still much stronger than Theorem \ref{theorem_in_diophantine_part}, since it applies to arbitrary lattice $\Gamma \subset G$.
\end{*remark}

\subsection{Extremity of submanifolds of matrix spaces}
\par Another direction to study Diophantine properties of a real matrix $\Phi \in \M(m\times n, \R)$ is to determine whether $\Phi$ is 
very well approximable. We say $\Phi \in \M(m\times n ,\R)$ is very well approximated if there exists some constant $\delta>0$ such that there
exist infinitely many nonzero integer vectors $\mathbf{p} \in \Z^n$ and integer vecotors $\mathbf{q}\in \Z^m$ such that
$$\|\Phi \mathbf{p} - \mathbf{q}\| \leq \|\mathbf{p}\|^{-n/m-\delta}.$$
A submanifold $\mathcal{U} \subset \M(m\times n,\R)$ is called extremal if with respect to the 
Lebesgue measure on $U$, almost every point is not very well approximated. Based the same correspondence as above, this type of problem can also be studied
from homogenous dynamics. Kleinbock and Margulis ~\cite{Klein_Mar} proved that if a submanifold $\mathcal{U} \subset \M(1\times n ,\R)$ 
is nondegenerate, then $\mathcal{U}$ is extremal.  Kleinbock, Margulis and Wang ~\cite{Klein_Mar_Wang} later gave a necessary 
and sufficient condition of a submanifold of $\M(m\times n, \R)$ being extremal. Recently, Aka, Breuillard, Rosenzweig and de Saxc\'{e} ~\cite{ABRS} gave a family of 
subvarieties of $\M(m\times n, \R)$, and announced a theorem stating that if a submanifold $\mathcal{U} \subset \M(m\times n , \R)$ is 
not contained in any one of the subvarieties given above, then $\mathcal{U}$ is extremal. It turns out that condition \ref{2nd} in 
Theorem \ref{theorem_in_diophantine_part} is stronger than the condition given in ~\cite{ABRS}. We will discuss it in detail in Appendix 
\ref{App:AppendixA}.


\par The article is organized as follows:
\begin{enumerate}
 \item In Section 2, we will follow the argument developed in ~\cite{Shah_1} to show any limit measure $\mu_{\infty}$ of
 $\{\mu_t: t>0\}$ is a probability measure and is invarant under some unipotent subgroup.
 \item In Section 3, we will apply Ratner's theorem and linearization technique to show that if the limit measure $\mu_{\infty}$ 
 is not $\mu_G$, then the curve $\varphi(I)$ must satisfy some linear algebraic condition concerning a particular finitely dimensional 
 representation $V$ of $H= \mathrm{SL}(2n,\R)$.
 \item In Section 4, we will complete the proof of Theorem \ref{goal_thm}. The proof is based on the linear algebraic condition we 
 get in Section 3, and a technical lemma proved in ~\cite{Yang_1} concerning the representations of $\mathrm{SL}(2,\R)$.
\end{enumerate}
\begin{notation}
 In this article, we will use the following notations: for $\epsilon > 0$ small, 
 and two quantities $A$ and $B$, $A \overset{\epsilon}{\approx} B$ means that 
 $|A-B| \leq \epsilon$. Fix a right $G$-invariant metric $d(\cdot, \cdot)$ on $G$, then for $x_1, x_2 \in G/\Gamma$,
 and $\epsilon >0$, $x_1 \overset{\epsilon}{\approx} x_2$ means $x_2 =g x_1$ such that $d(g,e)< \epsilon$. Given some quantity $A >0$, we denote by $O(A)$ some quantity $B$ such that $|B| \leq C A$ for some contant $C >0$. 
\end{notation}
\noindent {\bf Acknowledgement: }I would like to express my deep gratitude to my advisor, Professor Nimish Shah, for suggesting this problem to me, and 
his continuous advise and support during the process of the work. I also would like to thank Professor Kleinbock for reading 
an earlier version of this article and giving me a lot of comments and suggestions, for example, drawing ~\cite{Klein_Mar_Wang}
and ~\cite{ABRS} to my attention and suggesting me to figure out the relation between the geometric conditions given in this 
article and the arithmetic condition given in ~\cite{ABRS}.
\par Thanks are due to the referee for many useful suggestions.
\section{Non-divergence of limit measures and unipotent invariance}
\subsection{Preliminaries on Lie group structures}
\label{preliminaries}
At first we recall some basic facts about the group $H= \mathrm{SL}(2n,\R)$.
\par 
Let $A \subset H$ denote the diagonal subgroup as before, and let $Z_H(A)$ denote the centraliser of $A$ in $H$. Then
$$Z_H(A)= \left\{\begin{bmatrix}B & \\ & C\end{bmatrix}: B , C \in \mathrm{GL}(n,\R), \text{ and } \det B \det C =1\right\}.$$
\par Let $U^{+}(A)$ denote the expanding horospherical subgroup of $H$ with respect to the conjugate action of $A$, i.e., 
$$U^+ (A) := \{h \in H: a(-t) h a(t) \rightarrow e \text{ as } t \rightarrow +\infty\}.$$
Let $U^-(A)$ denote the contracting horospherical subgroup of $H$ with respect to the action of $A$ defined similarly. It is easily seen that 
$$U^{+}(A) = \left\{ u(X) = \begin{bmatrix}\I_n & X \\ & \I_n \end{bmatrix}: X \in \M(n\times n, \R)\right\},$$
and
$$U^{-}(A) = \left\{u^{-}(X) = \begin{bmatrix}\I_n & \\ X & \I_n\end{bmatrix}: X \in \M(n\times n, \R)\right\}.$$
For $z \in Z_H(A)$ and $u(X) \in U^{+}(A)$, the conjugate $z u(X) z^{-1}$ is still in $U^{+}(A)$. Let us denote $z u(X) z^{-1} = u(z \cdot X)$. Then it defines an action of 
$Z_H(A)$ on $U^{+}(A)$. It is easy to check that if
$$z = \begin{bmatrix}B & \\ & C\end{bmatrix},$$
then $z\cdot X = BXC^{-1}$. Similarly we can define the action of $Z_H(A)$ on $U^{-}(A)$.
\par For any $X \in \mathrm{GL}(n,\R)$, we consider the following three elements in the Lie algebra $\mathfrak{h}$ of $H$: 
$$\mathfrak{n}^{+}(X) := \begin{bmatrix} \mathbf{0} & X \\ \mathbf{0} & \mathbf{0} \end{bmatrix}, $$
$$\mathfrak{n}^{-}(X) : = \begin{bmatrix} \mathbf{0} & \mathbf{0} \\ X^{-1} & \mathbf{0}\end{bmatrix},$$
and 
$$\mathfrak{a} := \begin{bmatrix}\I_n & \mathbf{0} \\ \mathbf{0} & -\I_n\end{bmatrix}.$$
Then $\{\mathfrak{n}^{+}(X), \mathfrak{n}^{-}(X), \mathfrak{a}\}$ makes a $\mathfrak{sl}(2,\R)$ triple. Therefore, there is a embedding 
of $\mathrm{SL}(2,\R)$ into $H$ that sends $\begin{bmatrix}1 & 1 \\ 0 & 1\end{bmatrix}$ to $\exp(\mathfrak{n}^{+}(X))$, 
$\begin{bmatrix}1 & 0 \\ 1 & 1\end{bmatrix}$ to $\exp(\mathfrak{n}^{-}(X))$, and $\begin{bmatrix}e^t & 0 \\ 0 & e^{-t}\end{bmatrix}$ 
to $\exp(t \mathfrak{a}) = a(t)$. We call the image of this $\mathrm{SL}(2,\R)$ embedding $\mathrm{SL}(2,X) \subset H$. Let us denote 
$$\sigma(X) := \begin{bmatrix} & -X \\ X^{-1} &  \end{bmatrix} \in \mathrm{SL}(2,X),$$
it is easy to see that $\sigma(X)$ corresponds to $\begin{bmatrix}0 & -1 \\ 1 & 0\end{bmatrix} \in \mathrm{SL}(2,\R)$.
\subsection{Unipotent invariance of limit measures}
\par Recall that for $t >0$, $\mu_t$ denotes the normalized Lebesgue measure on the curve $a(t) u(\varphi(I))x$, and $\mu_G$ denote the unique $G$ invariant probability measure on $G/\Gamma$. Our aim is to prove that as $\mu_t \rightarrow \mu_G$ as $t \rightarrow +\infty$. However, due to a technical reason, it is hard to prove $\mu_t \rightarrow \mu_G$ directly. Instead, we 
need to modify the measures $\mu_t$ to another measure $\lambda_t$. We at first prove if $\lambda_t \rightarrow \mu_G$,
then $\mu_t \rightarrow \mu_G$, and then prove that $\lambda_t \rightarrow \mu_G$.
\par The measure $\lambda_t$ is defined as follows:
\begin{definition}
 For $t>0$ and a subinterval $J \subset I$, suppose that the derivative $\varphi^{(1)}(s) \in \mathrm{GL}(n,\R)$ for all $s \in J$.
 We define an analytic curve $z: J \rightarrow Z_H(A)$ such that $z(s)\cdot\varphi^{(1)}(s) = \I_n$ for all $s \in J$. Then we define 
 $\lambda^{J}_t$ to be the normalized Lebesgue measure on $\{z(s)a(t)u(\varphi(s))x: s \in J\}$, i.e., for $f\in C_c(G/\Gamma)$,
 $$\int f \dd \lambda^{J}_t : = \frac{1}{|J|}\int_{s\in J} f(z(s)a(t)u(\varphi(s))x)\dd s.$$
 Since $\varphi$ is analytic and satisfies condition \ref{1st} in Theorem \ref{theorem_in_diophantine_part}, 
 there are at most finitely many points where $\varphi^{(1)}(s)$ is not invertible. Thus we could cut $I$ into several open subintervals 
 $J_1, J_2, \dots, J_k$, such that $\varphi^{(1)}(s)$ is invertible for $s$ in any of these subintervals. Then we define the measure 
 $\lambda_t$ to be 
 $$\lambda_t := \sum_{i=1}^k \frac{|J_i|}{|I|} \lambda^{J_i}_t.$$
\end{definition}
\begin{*remark}
\par $\quad$
 
  \par 1. The above modification is due to Nimish Shah ~\cite{Shah_1} and ~\cite{Shah_2}.
  \par 2. For a subinterval $J\subset I$, we similarly define $\mu^{J}_t$ to be the normalized Lebesgue measure on 
  $a(t)u(\varphi(J))x$.

\end{*remark}
\begin{proposition}
\label{prop_lambda_mu}
 Suppose for any $J\subset I$ where $\lambda^{J}_t$ is defined, i.e., $\varphi^{(1)}(s) \in \mathrm{GL}(n,\R)$ for all $s \in J$,   we have $\lambda^{J}_t \rightarrow \mu_G$ as $t \rightarrow +\infty$.
 Then $\mu_t \rightarrow \mu_G$ as $t \rightarrow +\infty$.
\end{proposition}
\begin{proof}
 For any fixed $f \in C_c(G/\Gamma)$ and $\epsilon >0$, since $f$ is uniformly continuous, there exists a constant $\delta>0$, such that if $x_1 \overset{\delta}{\approx} x_2$
 then $f(x_1) \overset{\epsilon}{\approx} f(x_2)$.
 \par We cut $I$ into several small open subintervals $J_1, J_2, \dots, J_l$ such that $\lambda^{J_i}_t$ is defined for all $i=1,2,\dots, l$,
 and also, for every $J_i$, $z^{-1}(s_1)z(s_2) \overset{\delta}{\approx} e$ for any $s_1,s_2 \in J_i$. 
 \par Now for a fixed $J_i \subset I$, we choose $s_0 \in J_i$ and define $f_0(x)= f(z^{-1}(s_0)x)$. Then for any $s \in J_i$, because 
 $z^{-1}(s_0)z(s)a(t)u(\varphi(s))x \overset{\delta}{\approx} a(t)u(\varphi(s))x$, we have
 $$f_0(z(s)a(t)u(\varphi(s))x) = f(z^{-1}(s_0)z(s)a(t)u(\varphi(s))x) \overset{\epsilon}{\approx} f(a(t)u(\varphi(s))x).$$
 Therefore 
 $$\int f_0 d \lambda^{J_i}_t \overset{\epsilon}{\approx} \int f d \mu^{J_i}_t.$$
 Because $\int f_0 d \lambda^{J_i}_t \rightarrow \int_{G/\Gamma} f_0(x) d\mu_G(x)$ as $t \rightarrow +\infty$, and 
 $\int_{G/\Gamma} f_0(x) d\mu_G(x) = \int_{G/\Gamma}f(z^{-1}(s_0)x) d \mu_{G}(x) = \int_{G/\Gamma} f(x) d\mu_G$ (because $\mu_G$ is
 $G$-invariant), we have that there exists a constant $T_i>0$, such that for $t> T_i$, 
 $$\int f_0 d\lambda^{J_i}_t \overset{\epsilon}{\approx} \int_{G/\Gamma} f d\mu_G.$$
 Therefore 
 $$\int f d\mu^{J_i}_t \overset{2\epsilon}{\approx} \int_{G/\Gamma} f d\mu_G,$$
 for $t > T_i$.
 Then for $t > \max_{1\leq i\leq l} T_i$, 
 $$\int f d\mu_t \overset{2\epsilon}{\approx} \int_{G/\Gamma} f d \mu_G.$$
 Because $\epsilon>0$ can be arbitrarily small, we complete the proof.
\end{proof}
\begin{*remark}
By Proposition \ref{prop_lambda_mu}, to prove $\mu_t \rightarrow \mu_G $ as $t \rightarrow +\infty$, it suffices to show that for any subinterval $J \subset I$, $\lambda^{J}_t \rightarrow \mu_G$ as $t \rightarrow +\infty$. In particular, if we could prove the equidistribution of $\{\lambda_t : t >0\}$ as $t \rightarrow +\infty$ assuming $\varphi^{(1)}(s)$ is invertible everywhere, then the equidistribution of $\{\mu_t : t >0\}$ as $t \rightarrow +\infty$ will follow. Therefore, for the latter part of this paper, we will assume that $\varphi^{(1)}(s)$ is invertible for all $s \in I$ and thus $\lambda_t$ is defined to be the normalised Lebesgue measure on the curve $\{z(s)a(t)u(\varphi(s))x: s \in I\}$. Our goal is to show that $\lambda_t \rightarrow \mu_G$ as $t \rightarrow +\infty$.
\end{*remark}
The reason we modify $\mu_t$ to $\lambda_t$ is that it can be easily shown that any limit measure of $\{\lambda_t : t >0\}$ is 
invariant under the unipotent subgroup $W= \{u(t\I_n): t \in \R\}$.
\begin{proposition}[See ~\cite{Shah_1}]
 \label{prop_invariant_under_unipotent}
 Let $t_i \rightarrow +\infty$ be a sequence such that $\lambda_{t_i} \rightarrow \mu_{\infty}$ in weak-$\ast$ topology, then 
 $\mu_{\infty}$ is invariant under $W$-action.
\end{proposition}
\begin{proof}
Given any $f\in C_c(G/\Gamma)$, and $r \in \R$, we have 
$$
\int f(u(r\I_n)x)d\mu_{\infty}  =  \lim_{t_i \rightarrow +\infty} \frac{1}{|I|}\int_{s \in I} f(u(r\I_n)z(s)a(t_i)u(\varphi(s))x) ds.
$$
We want to argue that 
$$u(r\I_n)z(s)a(t_i)u(\varphi(s)) \overset{O(e^{-t_i})}{\approx} z(s+r e^{-t_i})a(t_i)u(\varphi(s+r e^{-t_i})).$$
Since $z(s+r e^{-t_i}) \overset{O(e^{-t_i})}{\approx} z(s)$ for $t_i$ large enough, it suffices to show that 
$$u(r\I_n)z(s)a(t_i)u(\varphi(s)) \overset{O(e^{-t_i})}{\approx} z(s)a(t_i)u(\varphi(s+r e^{-t_i})).$$
In fact,
$$\begin{array}{cl}
   & z(s)a(t_i)u(\varphi(s+r e^{-t_i})) \\
   = & z(s) a(t_i)u(\varphi(s) + r e^{-t_i} \varphi'(s) + \frac{r^2}{2} e^{-2 t_i} \varphi^{(2)}(s')) \\
   = & z(s)u(r \varphi'(s)) u(\frac{r^2}{2} e^{- t_i} \varphi^{(2)}(s')) a(t_i) u(\varphi(s)).
  \end{array}
$$
By the definition of $z(s)$, we have the above is equal to
$$ u(\frac{r^2}{2} e^{-t_i} z(s)\cdot\varphi^{(2)}(s'))u(r \I_n)z(s)a(t_i)u(\varphi(s)).$$
 This shows that  
$$u(r\I_n)z(s)a(t_i)u(\varphi(s)) \overset{O(e^{-t_i})}{\approx} z(s+ r e^{-t_i})a(t_i)u(\varphi(s+r e^{-t_i})).$$
Therefore, for any $\delta >0$, there is some constant $T >0$, such that for $t_i \geq T$,
$$u(r\I_n)z(s)a(t_i)u(\varphi(s)) \overset{\delta}{\approx} z(s+ r e^{-t_i})a(t_i)u(\varphi(s+r e^{-t_i})).$$
Given $\epsilon >0$, we choose $\delta>0$ such that whenever $x_1 \overset{\delta}{\approx} x_2$, we have 
$f(x_1) \overset{\epsilon}{\approx} f(x_2)$. Let $T>0$ be the constant as above. Then from the above argument, for $t_i > T$, we have
$$f(u(r \I_n) z(s) a(t_i)u(\varphi(s))x) \overset{\epsilon}{\approx} f(z(s+ r e^{-t_i}) a(t_i) u(\varphi(s+ r e^{-t_i}))x),$$
therefore,
$$\begin{array}{cl} & \frac{1}{|I|}\int_{s\in I} f(u(r \I_n) z(s) a(t_i)u(\varphi(s))x)\dd s \\
                  \overset{\epsilon}{\approx} & \frac{1}{|I|} \int_{s\in I} f(z(s+ r e^{-t_i}) a(t_i) u(\varphi(s+ r e^{-t_i}))x) \dd s \\
                  = & \frac{1}{|I|} \int_{a+ r e^{-t_i}}^{b+ r e^{-t_i}} f(z(s)a(t_i)u(\varphi(s))x) \dd s.
\end{array}$$
It is easy to see that when $t_i$ is large enough,
$$\frac{1}{|I|} \int_{a+ r e^{-t_i}}^{b+ r e^{-t_i}} f(z(s)a(t_i)u(\varphi(s))x) \dd s \overset{\epsilon}{\approx} \frac{1}{|I|} \int_{a}^{b} f(z(s)a(t_i)u(\varphi(s))x) \dd s.$$
Therefore, for $t_i$ large enough,
$$\int f(u(r \I_n)x) \dd \lambda_{t_i} \overset{2\epsilon}{\approx} \int f(x) \dd \lambda_{t_i}.$$
Letting $t_i \rightarrow +\infty$, we have 
$$\int f(u(r \I_n)x) \dd \mu_{\infty} \overset{2\epsilon}{\approx} \int f(x) \dd \mu_{\infty}.$$
Since the above approximation is true for arbitrary $\epsilon >0$, we have that 
$\mu_{\infty}$ is $W$-invariant.
\end{proof}
\subsection{Non-divergence of limit measures} We will prove that any limit measure $\mu_{\infty}$ of $\{\lambda_t: t>0\}$ is still a probability measure 
of $G/\Gamma$, i.e., no mass escapes to infinity. 
\par To show the non-divergence of limit measures, it suffices to show the following proposition:
\begin{proposition}[see ~\cite{Shah_1}]
\label{prop_no_escape_mass}
 For any $\epsilon >0$, there exists a compact subset $\mathcal{K}_{\epsilon} \subset G/\Gamma$ such that 
 $\lambda_t(\mathcal{K}_{\epsilon}) \geq 1-\epsilon$ for all $t >0$.
\end{proposition}
\par The proof of the proposition is due to Nimish Shah ~\cite{Shah_1}. Here we just modify the proof to fit our needs.
\begin{definition}
\label{def_representation_G}
 Let $\mathfrak{g}$ denote the Lie algebra of $G$, and denote $d = \dim G$. We define
 $$V = \bigoplus_{i=1}^{d} \bigwedge^i \mathfrak{g},$$
 and let $G$ act on $V$ via $\bigoplus_{i}^d \bigwedge^i \mathrm{Ad}(\cdot)$. This defines a linear representation of $G$:
 $$G \rightarrow \mathrm{GL}(V).$$
\end{definition}
The following theorem due to Kleinbock and Margulis ~\cite{Klein_Mar} is the basic tool to prove the non-divergence of limit measures:
\begin{theorem}[see ~\cite{Dani} and ~\cite{Klein_Mar}]
 \label{non_divergence_theorem}
 Fix a norm $\|\cdot \|$ on $V$. There exist finitely many vectors  $v_1, v_2, \dots , v_r \in V$ such that for each 
 $i=1,2,\dots, r$, the orbit $\Gamma v_i$ is discrete, and the
following holds: for any $\epsilon >0$ and $R > 0$, there exists a
compact set $K\subset G/\Gamma$ such that for any $t >0$ and any subinterval $J\subset I$, one of the
following holds:
\begin{enumerate}[label=\textbf{S.\arabic*}]
\item There exist $\gamma \in \Gamma$ and $j\in \{1,\dots , r\}$
such that $$\sup_{s\in J} \| a(t)u(\varphi(s)) g \gamma v_j \| < R.$$
\item $$|\{ s\in J:  a(t)u(\varphi(s))x \in K\}| \geq (1-\epsilon)|J|.$$
\end{enumerate}
\end{theorem}
\begin{*remark}
 The proof for polynomial curves is due to Dani ~\cite{Dani}, 
 the proof for analytic curves is due to Kleinbock and Margulis ~\cite{Klein_Mar}. The crucial part to prove the above theorem is to find constants 
 $C>0$ and $\alpha>0$ such that in this particular representation, all the coordinate functions of $a(t)u(\varphi(\cdot))$ 
 are $(C, \alpha)$-good. Here a function $f: I \rightarrow \R$ is called
 $(C,\alpha)$-good if for any subinterval $J \subset I$ and any $\epsilon >0$, the following holds:
 $$|\{s\in J: |f(s)|<\epsilon\}| \leq C\left(\frac{\epsilon}{\sup_{s\in J}|f(s)|}\right)^{\alpha} |J|.$$
\end{*remark}

\begin{notation}
 Let $F$ be a Lie group, and $V$ be a finite dimensional linear representation of $F$. Then for a one-parameter diagonal subgroup 
 $A=\{a(t): t \in \R\}$ of $F$, we could decompose $V$ as direct sum of eigenspaces of $A$, i.e.,
 $$V = \bigoplus_{\lambda \in \R} V^{\lambda}(A),$$
 where $V^{\lambda}(A) = \{v\in V: a(t)v = e^{\lambda t} v\}$.
 \par We define
 $$V^{+}(A) = \bigoplus_{\lambda >0} V^{\lambda}(A),$$
 $$V^{-}(A) = \bigoplus_{\lambda <0 } V^{\lambda}(A),$$
 and similarly,
 $$V^{+0}(A) = V^{+}(A) + V^0(A),$$
 $$V^{-0}(A) = V^{-}(A) + V^0(A) .$$
 For a vector $v \in V$, we denote by $v^{+}(A)$ ($v^{-}(A)$, $v^0(A)$, $v^{+0}(A)$ and $v^{-0}(A)$ respectively) the projection of $v$ onto $V^{+}(A)$ ($V^{-}(A)$, $V^0(A)$, $V^{+0}(A)$ and 
 $V^{-0}(A)$ respectively).
\end{notation}

\par The following basic lemma on representations of $\mathrm{SL}(2,\R)$ due to Nimish Shah is crucial in the proof of Proposition \ref{non_divergence_theorem}:
\begin{lemma}{(See ~\cite[Lemma 2.3]{Shah_1})}
\label{lemma_non_divergent}
\par Let $V$ be a representation of $\mathrm{SL}(2,\R)$, fix a norm $\|\cdot\|$ on $V$.
We define 
$$A = \left\{a(t)= \begin{bmatrix}e^{t} & \\ & e^{-t}\end{bmatrix}: t \in \R \right\},$$
and
$$U^{+}(A) = \left\{u(t)=\begin{bmatrix}1 & t \\ 0 & 1 \end{bmatrix}: t \in \R \right\}.$$
Then for any $t >0$, there exists a
constant $\kappa=\kappa(t)> 0$ such that for any $v \in V$,
$$
\max\{\|v^+(A)\|,\|(u(t)v)^{+0}(A)\|\} \geq \kappa \|v\|.
$$
\end{lemma}
\par In $H= \mathrm{SL}(2n,\R)$, for any $X \in \mathrm{GL}(n,\R)$, $u(X) \in \mathrm{SL}(2,X) \subset H$ corresponds to 
$\begin{bmatrix}1 & 1 \\ 0 & 1\end{bmatrix}$ in $\mathrm{SL}(2,\R)$, where 
$\mathrm{SL}(2,X) \cong \mathrm{SL}(2,\R)$ is defined in Subsection \ref{preliminaries}. Lemma \ref{lemma_non_divergent} easily implies the following:
\begin{corollary}{(See ~\cite[Corollary 2.4]{Shah_1})}
\label{corollary_of_lemma_nondivergent}
 Let $V$ be a linear representation of $H=\mathrm{SL}(2n,\R)$, fix a norm $\|\cdot\|$ on $V$. Let $A=\{a(t): t \in \R\} \subset H$ be the 
 one-parameter diagonal subgroup as in Section 1.
Then given a compact set $\mathcal{F}\subset \mathrm{GL}(n,\R)$, there exists
a constant $\kappa >0$ such that for any $ X \in \mathcal{F}$ and any
$v\in V$,
$$
\max \{\|v^+(A)\|,\|(u(X)v)^{+0}(A)\|\}\geq \kappa \|v\|,
$$
In particular, for any $t >0$, any $X\in \mathcal{F}$ and any
$v\in V$,
$$
\max \{\|a(t)v\|,\|a(t)u(X)v\|\} \geq \kappa \|v\|.
$$
\end{corollary}

\begin{proof}[Proof of Proposition \ref{prop_no_escape_mass}]
  Fix $s_1 \in I$ and subinterval $J_{s_0}$ such that $\varphi(s)- \varphi(s_1) \subset \mathrm{GL}(n,\R)$ for $s \in J_{s_0}$, and a compact subset $\mathcal{F} \subset \mathrm{GL}(n,\R)$ containing 
  $\{\varphi(s)-\varphi(s_1): s \in J_{s_0}\}$ 
  and let $\kappa >0$ be the constant provided 
  in Corollary \ref{corollary_of_lemma_nondivergent} with respect to $\mathcal{F}$. 
\par Now for any $\epsilon>0$ and $R >0$, by Theorem \ref{non_divergence_theorem}, there exists a 
 compact subset $K\subset G/\Gamma$, such that for any $t >0$, one of the
following holds:
\begin{enumerate}[label=\textbf{S.\arabic*}]
\item \label{divergence_condition} There exist $\gamma \in \Gamma$ and $j\in \{1,\dots , r\}$
such that $$\sup_{s\in I} \| a(t)u(\varphi(s)) g \gamma v_j \| < R.$$
\item \label{non_divergence_condition} $$|\{ s\in I:  a(t)u(\varphi(s))x \in K\}| \geq (1-\epsilon)|I|.$$
\end{enumerate}
 Now fix $s_2 \in J$ and denote $X = \varphi(s_2) - \varphi(s_1)$. 
 Then because $\Gamma v_i$ is discrete in $V\setminus\{\mathbf{0}\}$, there exists
 a uniform constant $r>0$ such that 
 $$\|u(\varphi(s_1))g \gamma v_i\| \geq r,$$
 for any $v_i$ and $\gamma \in \Gamma$. Applying Corollary \ref{corollary_of_lemma_nondivergent} with $v$ replaced by $u(\varphi(s_1))g \gamma v_i$, 
 we get for any $v_i$, $\gamma \in \Gamma$ and $t > 0$,
 $$\sup_{s \in I} \|a(t)u(\varphi(s))g\gamma v_i\| \geq \kappa r.$$
 If we choose $R < \kappa r$, then case \ref{divergence_condition} above can not hold, 
 this shows that 
 $$|\{s\in I: a(t)u(\varphi(s))x \in K\}| \geq (1-\epsilon)|I|.$$
 Let $\mathcal{K}_{\epsilon} = MK$, since $z(s)\in M$, we have 
 $$|\{s\in I: z(s)a(t)u(\varphi(s))x \in \mathcal{K}_{\epsilon}\}| \geq (1-\epsilon)|I|,$$
 i.e., $\lambda_t(\mathcal{K}_{\epsilon}) \geq 1-\epsilon$ for all $t >0$.
 \par This completes the proof.
\end{proof}
\begin{*remark} Proposition \ref{prop_no_escape_mass} implies that any limit measure $\mu_{\infty}$ of $\{\lambda_t : t >0\}$ 
is still a probability measure. In fact, suppose $\lambda_{t_i} \rightarrow \mu_{\infty}$ along some subsequence $t_{i} \rightarrow +\infty$. Then by Proposition \ref{prop_no_escape_mass}, for any $\epsilon >0$, there exists a compact subset $\mathcal{K} \subset G/\Gamma$ such that $\lambda_t(\mathcal{K}) > 1 - \epsilon$ for all $t >0$. Thus, $\mu_{\infty}(G/\Gamma) \geq \mu_{\infty}(\mathcal{K}) > 1 - \epsilon$. Letting $\epsilon \rightarrow 0$, we get 
$\mu_{\infty}(G/\Gamma) \geq 1$.
\end{*remark}
\section{Ratner's theorem and linearization technique}
Take any convergent subsequence $\lambda_{t_i} \rightarrow \mu_{\infty}$. By Proposition \ref{prop_invariant_under_unipotent} and Proposition 
\ref{prop_no_escape_mass}, $\mu_{\infty}$ is a $W$-invariant probability measure on $G/\Gamma$. 
\par In order to apply Ratner's theorem and linearization technique, we need to introduce some notations at first.
\begin{definition}
 Let $\mathcal{L}$ be the collection of analytic subgroups $L < G$ such that $L\cap \Gamma$ is a lattice of $L$. one can prove that
$\mathcal{L}$ is a countable set (see ~\cite{Ratner}).
\par For $L\in \mathcal{L}$, define:
$$N(L,W):= \{g\in G: g^{-1}Wg\subset L\},$$
and 
$$S(L,W):= \bigcup_{L'\in \mathcal{L}, L' \subsetneq L} N(L', W).$$
\end{definition}
\par We recall the Ratner's measure classification theorem as follows:
\begin{theorem}[See ~\cite{Ratner}]
\label{ratner} Given the $W$-invariant probability measure $\mu$ on
$G/\Gamma$, there exists $L \in \mathcal{L}$ such that
\begin{equation}
\begin{array}{ccc}
\mu(\pi(N(L,W)))>0 & \text{ and } & \mu(\pi(S(L,W)))=0
\end{array}
\end{equation}
Moreover, almost every $W$-ergodic component of $\mu$ on
$\pi(N(L,W))$ is a measure of the form $g\mu_L$ where $g\in
N(L,W)\backslash S(L,W)$, $\mu_L$ is a finite $L$-invariant measure
on $\pi(L)$, and $g\mu_L(E)=\mu_L (g^{-1}E)$ for all Borel sets
$E\subset G/\Gamma$. In particular, if $L\lhd  G$, then $\mu$ is
$L$-invariant.
\end{theorem}
\par If $\mu_{\infty}=\mu_G$, then there is nothing to prove. So we may assume $\mu_{\infty}\neq \mu_G$. Then by Ratner's Theorem, 
there exists $L \in \mathcal{L}$ such that $\mu_{\infty}(\pi(N(L,W))) > 0$ and $\mu_{\infty}(\pi(S(L,W))) =0$. Now we start to apply the 
linearization technique.
\par We start with some basic notations:
\begin{definition}
Let $V$ be the finitely dimensional representation of $G$ defined as in Definition \ref{def_representation_G}, for $L \in \mathcal{L}$,
we choose a basis $\mathfrak{e}_1, \mathfrak{e}_2, \dots, \mathfrak{e}_l$ of the Lie algebra $\mathfrak{l}$ of $L$, and define 
$$p_L = \wedge_{i=1}^l \mathfrak{e}_i \in V.$$
Define 
$$\Gamma_L := \left\{\gamma\in \Gamma: \gamma p_L = \pm p_L \right\}.$$
From the action of $G$ on $p_L$, we get a map:
$$\begin{array}{l} \eta:  G \rightarrow V,  \\
                     g \mapsto g p_L .
   \end{array}
$$
We define $\mathcal{A}$ to be the Zariski closure of $\eta(N(L,W))$. and for any compact subset $\mathcal{D} \subset \mathcal{A}$, we define
$$S(\mathcal{D}) := \left\{g \in N(L,W):  \eta(g\gamma) \in \mathcal{D} \text{ for some } \gamma \in \Gamma\setminus \Gamma_L\right\}.$$
\end{definition}
Concerning $S(\mathcal{D})$, we have the following important propositions:
\begin{proposition}[see Proposition 4.5 of ~\cite{Shah_1}]
\label{injective_prop}
 $S(\mathcal{D}) \subset S(L,W)$ and
$\pi(S(\mathcal{D}))$ is closed in $G/\Gamma$. Moreover, for any
compact set $\mathcal{K} \in G/\Gamma \backslash
\pi(S(\mathcal{D}))$, there exists some neighborhood $\Phi$ of
$\mathcal{D}$ in $V$ such that, for any $g\in G$ and $\gamma_1,
\gamma_2 \in \Gamma$, if $\pi(g)\in \mathcal{K}$ and $\eta(g
\gamma_i) \in \Phi$, $i=1,2$, then $\eta(\gamma_1)=\pm
\eta(\gamma_2)$.
\end{proposition}
\begin{proposition}[see Proposition 4.6 of ~\cite{Shah_1}]
\label{relative_small prop} Given a symmetric compact set
$\mathcal{C}\subset\mathcal{A}$ and $\epsilon > 0$, there exists a
symmetric compact set $\mathcal{D}\subset \mathcal{A}$ containing
$\mathcal{C}$ such that, given a symmetric neighborhood $\Phi$ of
$\mathcal{D}$ in $V$, there exists a symmetric neighborhood $\Psi$
of $\mathcal{C}$ in $V$ contained in $\Phi$ such that for any $t >0$, 
for any $v \in V$, and for any interval
$J\subset I$, one of the following holds:
\begin{enumerate}[label=\textbf{SS.\arabic*}]
\item \label{1} $a(t)u(\varphi(s))v \in \Phi$ for all $s\in J$.
\item \label{2} $|\{s\in J: a(t)u(\varphi(s))v \in \Psi \}|\leq \epsilon |\{s\in J: a(t)u(\varphi(s))v \in \Phi\}|$.
\end{enumerate}
\end{proposition}
\begin{*remark}
 The proof is similar to Theorem \ref{non_divergence_theorem}, and also follows from the fact that all coordinate functions of
 $a(t)u(\varphi(\cdot))v$ are 
 $(C, \alpha)$-good for some constants $C >0$ and $\alpha >0$.
\end{*remark}
\par The following proposition is the 
aim of this section.
\begin{proposition}
 \label{prop_algebraic_condition}
 There exists a $\gamma \in \Gamma$ such that
 $$u(\varphi(s))g\gamma p_L \in V^{-0}(A),$$
 for all $s \in I$.
\end{proposition}

\begin{proof}
 Take a compact subset $C \subset N(L,W))\setminus S(L,W)$ such that $\mu_{\infty}(\pi(C)) >c_0 >0$ for some constant $c_0$. Define 
 $\mathcal{C} := \eta(C)\cup ( - \eta(C))$, then $\mathcal{C} \subset \mathcal{A}$ is a compact subset. Choose a 
 compact subset $\mathcal{K}\subset G/\Gamma\setminus \pi(S(L,W))$ containing $\pi(C)$ in its interior. Applying Proposition \ref{relative_small prop},
 we can find a symmetric compact subset $\mathcal{D} \subset \mathcal{A}$ containing $\mathcal{C}$ such that 
 the conclusion of Proposition \ref{relative_small prop} holds for $\mathcal{C}$, $\mathcal{D}$ and some small $0< \epsilon < \frac{c_0}{2}$. 
 Applying Proposition \ref{injective_prop} to 
 $\mathcal{D}$ and $\mathcal{K}$, we have that there exists an open neighborhood $\Phi$ of $\mathcal{D}$ such that the conclusion of Proposition 
 \ref{injective_prop} holds. 
 Choose a neighborhood
 $\Psi$ of $\mathcal{C}$ according to Proposition \ref{relative_small prop}. 
 \par We claim that there exists $\gamma_t \in \Gamma$ such that 
 $$a(t)u(\varphi(I))g \gamma_t p_L \subset \Phi.$$
 For contradiction, we assume it is not the case, i.e., 
 for all $\gamma \in \Gamma$, case \ref{1} in Proposition \ref{relative_small prop} does not hold for 
 $v = g\gamma p_L$ and $J=I$. We define
 $$J_t := \left\{ s\in I: a(t)u(\varphi(s))x \in \mathcal{K} : a(t)u(\varphi(s))g\Gamma p_L \cap \Psi \neq \emptyset  \right\},$$
 then for $t$ large enough, $|J_t| > c_0|I|$. 
 \par By Proposition \ref{injective_prop}, for any $s \in J_t$, , up to $\pm$ sign, there exists unique $\gamma(s) p_L$ such that 
 $a(t)u(\varphi(s))g\gamma(s) p_L \in \Psi$, let $I_{\gamma(s)}$ be the maximal interval $I$ containing $s$ such that 
 $$a(t)u(\varphi(I))g \gamma(s) p_L \subset \Phi.$$
 From Proposition \ref{injective_prop} we know that there is no other $\gamma' p_L$ other than $\pm \gamma(s) p_L$ 
 and $s \in I_{\gamma(s)}\cap J_t$ such that
 $$a(t)u(\varphi(s))g \gamma' p_L \in \Psi.$$
 Therefore $J_t$ is covered by at most countably many intervals $I_{\gamma(s)}$'s which covers the whole interval $I$ at most twice, namely,
 every point belongs to at most two different intervals (this is because for any $s_1< s_2 \in J_t$, then from the above argument, the intersection 
 $I_{\gamma(s_1)}\cap I_{\gamma(s_2)} \subset (s_1, s_2)$).
 Moreover, because case \ref{1} in Proposition \ref{relative_small prop} does not hold, we have that \ref{2} must hold, i.e., 
 $$|J_t \cap I_{\gamma(s)}| < \epsilon |I_{\gamma(s)}|.$$
 This shows that 
 $$|J_t| < 2\epsilon |I|$$
 which contradicts to the fact that $|J_t| > c_0|I|$. This shows the claim.
 \par Since $\Gamma p_L$ is discrete in $V$, one of the following will happen:
 \begin{enumerate}
  \item $\|\gamma_t p_L\| \rightarrow +\infty$ as $t \rightarrow \infty$.
  \item $\gamma_t p_L$ remains the same for all large $t$.
 \end{enumerate}
If case 1 happens, define a unit vector $v_t = \frac{\gamma_t p_L}{\|\gamma_t p_L\|}$ for each $t$, 
then from 
$$a(t)u(\varphi(I)) \gamma(t) p_L \subset \Phi$$
we have there is a constant $R$ such that 
$$\sup_{s \in I}\|a(t)u(\varphi(s))v_t\| \leq \frac{R}{\|\gamma_t p_L\|} \rightarrow 0.$$
Suppose $v_t \rightarrow v_{\infty}$ passing to some subsequence, then we have 
$$\sup_{s \in I} \|a(t)u(\varphi(s))v_{\infty}\| \rightarrow 0,$$
as $t \rightarrow +\infty$. This is impossible according to Corollary \ref{corollary_of_lemma_nondivergent} and condition \ref{2nd} in 
Theorem \ref{theorem_in_diophantine_part}. 
Therefore $\gamma_t p_L = \gamma p_L$ remains the same for all large $t$. This means that for all $t >0$,
$$\sup_{s \in I}\|a(t)u(\varphi(s))g \gamma p_L\| \leq R.$$
This implies that for $v = g \gamma p_L$,
$$u(\varphi(s))v \in V^{-0}(A).$$
\par This completes the proof.
\end{proof}

\section{Conclusion}
In this section we will finish the proof of Theorem \ref{goal_thm}. By the correspondence between homogeneous dynamics and Diophantine approximation
discussed in the introduction, Theorem \ref{theorem_in_diophantine_part} follows from Theorem \ref{goal_thm}.
\par We need the following basic lemma on $\mathrm{SL}(2,\R)$ representations proved in ~\cite{Yang_1}.
\begin{lemma}[See ~\cite{Yang_1}]
 \label{basic_lemma}
 Let $V$ be a finite dimensional linear representation of $\mathrm{SL}(2,\R)$. Denote 
 $$
 A := \left\{a(t) := \begin{bmatrix}e^t & \\ & e^{-t}\end{bmatrix}: t \in \R\right\},
 $$
 and 
 $$
 U := \left\{u(s) := \begin{bmatrix}1 & s \\ 0 & 1\end{bmatrix}\right\}.
 $$
 Suppose there is a nonzero vector $v \in V^{-0}(A)$ satisfying 
 $$u(r)v \in V^{-0}(A),$$
 for some $r \in \R$, then $(u(r)v)^0(A) = \sigma v^0(A),$
 where $\sigma$ denotes the matrix
 $$\sigma = \begin{bmatrix}0 & -1 \\ 1 & 0\end{bmatrix}.$$
\end{lemma}
\begin{proof}[Proof of Theorem \ref{goal_thm}]
We start with the linear algebraic condition we get in Proposition \ref{prop_algebraic_condition}:
$$u(\varphi(s))v \in V^{-0}(A),$$
for all $s \in I$.

\begin{claim}
 $(u(\varphi(s))v)^{0}(A)$ is invariant under the unipotent flow $\{u(r \varphi^{(1)}(s)): r \in \R\}$.
 
\end{claim}
\begin{proof}[Proof of the claim:]
On the one hand, since $u(\varphi(s))v \in V^{-0}(A)$ for all $s \in I$, for any fixed $r \in \R$, 
$$\lim_{t\rightarrow +\infty} a(t)u(\varphi(s+{r/e^{-2t}}))v = \lim_{t\rightarrow +\infty} (u(\varphi(s+{r/e^{-2t}}))v)^0(A) = (u(\varphi(s))v)^0(A).$$
On the other hand, 
$$
\begin{array}{cl}
 & \lim_{t \rightarrow +\infty} a(t)u(\varphi(s+{r/e^{-2t}}))v \\
 =& \lim_{t\rightarrow +\infty} a(t)u(\varphi(s)+ r\varphi^{(1)}/e^{-2t} + o(e^{-4t}))v \\
 = & \lim_{t\rightarrow + \infty} a(t)u(r\varphi^{(1)}/e^{-2t} + o(e^{-4t})a(-t) a(t)u(\varphi(s))v \\
 = & \lim_{t\rightarrow + \infty} u(r\varphi^{(1)}(s)+ o(e^{-2t})) a(t)u(\varphi(s))v \\
 = & u(r\varphi^{(1)}(s))(u(\varphi(s))v)^{0}(A).
\end{array}
$$
This implies $(u(\varphi(s))v)^0(A)=u(r\varphi^{(1)}(s))(u(\varphi(s))v)^{0}(A)$ for all $r \in \R$.
\par This proves the claim.
\end{proof}
Fix any $s_0 \in I$, since $\varphi^{(1)}(s_0)$ is invertible, there exists a subinterval $J_{s_0} \subset I$, such that 
$\varphi(s)-\varphi(s_0)$ is invertible for all $s \in J_{s_0}$. Let us denote $X(s)= \varphi(s)-\varphi(s_0)$, and consider the subgroup
$\mathrm{SL}(2,X(s))\cong \mathrm{SL}(2,\R)$. Notice that in $\mathrm{SL}(2,X(s))$, 
$$\sigma(X(s)) =\begin{bmatrix} & -X(s) \\ X^{-1}(s) & \end{bmatrix}$$
corresponds to 
$$\sigma = \begin{bmatrix} 0 & -1 \\ 1 & 0\end{bmatrix} \in \mathrm{SL}(2,\R).$$ 
Applying Lemma \ref{basic_lemma} with $\mathrm{SL}(2,\R)$ replaced by $\mathrm{SL}(2,X(s))$,
$v$ replaced by $u(\varphi(s_0))v$ and $u(r)$ replaced by $u(X(s))$, we conclude that 
$$(u(\varphi(s))v)^0(A) = \sigma(X(s))(u(\varphi(s_0))v)^0(A).$$
Let us denote $w = (u(\varphi(s_0))v)^0(A)$ and define 
$$\mathcal{S}:= \{X \in \M(n\times n, \R): u^{-}(X)w = w \}.$$
It is clear that $\mathcal{S}$ is a subspace of $\M(n\times n, \R)$.
\begin{claim} 
 $\mathcal{S}$ is a proper subspace of $\M(n\times n, \R)$.
\end{claim}
\begin{proof}[Proof of the claim:]
Suppose not, then 
$w$ is fixed by the whole horospherical subgroup $U^{-}(A)$. $w$ is also fixed by $A$ since it is in $V^0(A)$. This implies that 
$w$ is fixed by $H = \mathrm{SL}(2n,\R)$. Let us denote $u(\varphi(s_0))v = w + w^{-}$, where $w^{-} = u(\varphi(s_0))v - w \in V^{-}(A)$.
We claim that $w^{-} = \mathbf{0}$. In fact, if $w^{-} \neq \mathbf{0}$, then for $s \in J$,
$$u(X(s))w^{-} = u(\varphi(s))v - u(X(s))w = u(\varphi(s))v - w \in V^{-0}(A).$$
This contradicts to Lemma \ref{lemma_non_divergent} with $\mathrm{SL}(2,\R)$ replaced by $\mathrm{SL}(2,X(s))$, $u(t)$ replaced by $u(X(s))$, 
and $v$ replaced by $w^{-}$. This shows that $w^{-}=\mathbf{0}$. Thus, $u(\varphi(s_0))v = w$ is fixed by the whole 
group $H=\mathrm{SL}(2n,\R)$. Then $v =g\gamma p_L$ is fixed by $H$. Hence $p_L$
is fixed by the action of
$\gamma^{-1}g^{-1}Hg\gamma$. Thus
$$
\begin{array}{rcl}
\Gamma p_L & = & \overline{\Gamma p_L} \text{ since } \Gamma p_L
 \text{ is discrete} \\
 & = & \overline{\Gamma \gamma^{-1} g^{-1}Hg\gamma
 p_L} \\ & = & \overline{\Gamma g^{-1}Hg\gamma
 p_L} \\ & = & G g\gamma p_L \text{ since } \overline{Hg
 \Gamma}=G \\
  & = & G p_L.
 \end{array}
$$
This implies $G_0 p_L = p_L$ where $G_0$ is the connected component
of $e$. In particular, $\gamma^{-1} g^{-1}H g
\gamma\subset G_0$ and $G_0 \subset N^1_{G}(L)$. By ~\cite[Theorem
2.3]{Shah_3}, there exists a closed subgroup $F_1 \subset
N^1_{G}(L)$ containing all $\mathrm{Ad}$-unipotent one-parameter
subgroups of $G$ contained in $N^1_{G}(L)$ such that $F_1 \cap
\Gamma$ is a lattice in $F_1$ and $\pi (F_1)$ is closed. If we put
$F= g\gamma F_1 \gamma^{-1} g^{-1}$, then $H\subset
F$ since $H$ is generated by its unipotent
one-parameter subgroups. Moreover, $Fx =g \gamma \pi(F_1)$ is closed
and admits a finite $F$-invariant measure. Then since
$\overline{Hx}=G/\Gamma$, we have $F=G$. This implies
$F_1 = G$ and thus $L \lhd G$. Therefore $N(L,W)=G$. In particular,
$W\subset L$, and thus $L\cap H$ is a normal subgroup
of $H$ containing $W$. Since $H$ is a
simple group, we have that $H \subset L$. Since $L$ is a
normal subgroup of $G$ and $\pi(L)$ is a closed orbit with finite
$L$-invariant measure, every orbit of $L$ on $G/\Gamma$ is also
closed and admits a finite $L$-invariant measure, in particular,
$Lx$ is closed. But since $Hx$ is dense in
$G/\Gamma$, $Lx$ is also dense. This shows that $L=G$, which
contradicts to our hypothesis that $\mu_{\infty} \neq \mu_G$. This proves the claim.
\end{proof}
We fix an inner product $\langle , \rangle$ on $\M(n\times n, \R)$ and a nonzero vector $\mathbf{Y} \in \M(n\times n , \R)$ such that 
any $X \in \mathcal{S}$ satisfies $\langle X , \mathbf{Y} \rangle =0$.
\par We have proved that $(u(\varphi(s))v)^0(A)$ is fixed by $\{u(r\varphi^{(1)}(s)): r \in \R\}$. Therefore, $w= (u(\varphi(s_0))v)^{0}(A)$ is fixed by 
$$(\sigma(X(s)))^{-1}u(\varphi^{(1)}(s_0))\sigma(X(s))= \begin{bmatrix}\I_n &  \\ H(s) 
 & \I_n \end{bmatrix},$$
 where $H(s)= - X(s)^{-1} \varphi^{(1)}(s) X(s)^{-1} .$
 This means that $\langle - X(s)^{-1} \varphi^{(1)}(s) X(s)^{-1} , \mathbf{Y}\rangle =0$. Note that 
 $$ ((X(s))^{-1})^{(1)} = - X(s)^{-1} \varphi^{(1)}(s) X(s)^{-1}.$$ 
 This implies that $\langle (X(s))^{-1}, \mathbf{Y} \rangle$ is a constant, i.e., $\{X(s)^{-1} = (\varphi(s)- \varphi(s_0))^{-1}: s \in J\}$ is contained in a proper affine 
 subspace of $\M(n\times n, \R)$. Because this holds for arbitrary $s_0 \in I$, we get a contradiction to condition \ref{2nd} in 
 Theorem \ref{theorem_in_diophantine_part}.
 \par This completes the proof of Theorem \ref{goal_thm}.
\end{proof}

\appendix
\section{Relation between the condition given in ~\cite{ABRS} and condition \ref{2nd} in Theorem \ref{theorem_in_diophantine_part}} \label{App:AppendixA}
\par We will discuss the condition given in ~\cite{ABRS} and its relation with the 
condition \ref{2nd} in Theorem \ref{theorem_in_diophantine_part}. Because in this article we only consider the case 
$m=n$, we only discuss this special case here.
\par We denote $M(s) = [\I_n, \varphi(s) ] \in \M(n\times 2n,\R)$. Given a subspace $W$ and $0<r < \frac{\dim W}{2}$, we define 
the pencil $\mathcal{P}_{W,r}$ to be 
$$\mathcal{P}_{W,r} := \{M \in \M(n\times 2n,\R) : \dim M W =r\}.$$
In ~\cite{ABRS}, the following theorem is announced: if a submanifold is not contained in any of these pencils defined above,
then the submanifold is extremal. In our setup, it says that if the curve $\{[\I_n,\varphi(s)]: s \in I\}$ is not contained in any pencil 
$\mathcal{P}_{W,r}$, then the curve is extremal. It is easy to see that if $W$ is a rational subspace, then $\mathcal{P}_{W,r}$ is not extremal.
So this condition is almost optimal.
\begin{claim}
If the curve $\varphi(I)$ satisfies that for some $s_0 \in I$, $\varphi(s)-\varphi(s_0)$ is invertible for $s$ in a subinterval
$J$ of $I$, then the curve $\{[\I_n, \varphi(s)]: s \in I\}$ is not contained in any pencil $\mathcal{P}_{W,r}$.
\end{claim}
\begin{proof}[Proof of the claim:]
 Suppose not, then the curve $\{[\I_n, \varphi(s)]: s \in I\}$ is contained in some pencil $\mathcal{P}_{W,r}$ where $r < \frac{\dim W }{2}$.
 This means that if we denote $M(s) = [\I_n, \varphi(s)]$, the intersection of $W$ and the kernel of $M(s)$ has dimension greater than 
 $\frac{\dim W}{2}$, i.e., 
 $$\dim \mathrm{Ker} (M(s)) \cap W > \frac{\dim W}{2}.$$
 Let us denote $W(s) := \mathrm{Ker} (M(s)) \cap W$, then for $s_1\neq s_2 \in I$, $W(s_1)\cap W(s_2) \neq \{\mathbf{0}\}$ because the sum of 
 their dimensions is greater than $\dim W$. This means that the intersection $\mathrm{Ker} (M(s_1))\cap \mathrm{Ker} (M(s_2)) \neq \{\mathbf{0}\}$.
 It is easy the see that the kernel of $M(s)$ is $\{(-\varphi(s)w, w): w \in \R^n\}$, so there exist $w_1 , w_2 \in \R^n\setminus\{\mathbf{0}\}$ such that
 $$(-\varphi(s_1)w_1, w_1) = (-\varphi(s_2)w_2, w_2),$$
 this implies that $w_1=w_2$, and $-\varphi(s_1)w_1 = -\varphi(s_2)w_2$. Therefore $(\varphi(s_1)-\varphi(s_2))w_1 = \mathbf{0}$. This shows that 
 $\varphi(s_1)-\varphi(s_2)$ is not invertible, for any $s_1\neq s_2 \in I$. This gives a contradiction.
\end{proof}
This shows that condition ~\ref{2nd} is stronger than the condition given in ~\cite{ABRS}.

\bibliography{reference}{}
\bibliographystyle{plain}

\end{document}